\documentclass[a4paper]{amsart}
\usepackage{color}
\usepackage[latin1]{inputenc}
\usepackage[T1]{fontenc}
\usepackage{amsfonts}
\usepackage{amssymb}
\usepackage{amsmath}
\usepackage{amsthm}
\usepackage{stmaryrd}
\usepackage{eufrak}
\usepackage{graphicx,type1cm,eso-pic,color}
\usepackage{pstricks,pst-plot,pstricks-add}
\usepackage{float}
\newtheorem{theorem}{Theorem}[section]
\newtheorem{lemma}[theorem]{Lemma}
\newtheorem{proposition}[theorem]{Proposition}
\newtheorem{corollary}[theorem]{Corollary}

\newtheorem{assumption}[theorem]{Assumption}

\begin{document}
\setlength\arraycolsep{2pt}
\title{Nonparametric Estimation of the Trend in Reflected Fractional SDE}
\author{Nicolas MARIE*}
\address{*Laboratoire Modal'X, Universit\'e Paris Nanterre, Nanterre, France}
\email{nmarie@parisnanterre.fr}
\address{*ESME Sudria, Paris, France}
\email{nicolas.marie@esme.fr}
\keywords{Nonparametric estimation ; Trend estimation ; Skorokhod reflection problem ; Sweeping process ; Fractional Brownian motion ; Stochastic differential equations}
\date{}
\maketitle
\noindent
%

% Abstract.

%
\begin{abstract}
This paper deals with the consistency, a rate of convergence and the asymptotic distribution of a nonparametric estimator of the trend in the Skorokhod reflection problem defined by a fractional SDE and a Moreau sweeping process.
\end{abstract}
\tableofcontents
\noindent
\textbf{Acknowledgments.} This work was supported by the GdR TRAG. Many thanks to Paul Raynaud de Fitte for his advices.
%

% Section : Introduction.

%
\section{Introduction}\label{section_introduction}
Consider $T > 0$ and the Skorokhod reflection problem
\begin{equation}\label{main_equation}
\left\{
\begin{array}{rcl}
 X_{\varepsilon}(t) & = & \displaystyle{\int_{0}^{t}b(X_{\varepsilon}(s))ds} +\varepsilon B(t) + Y_{\varepsilon}(t)\\
 -\dot Y_{\varepsilon}(t) & \in & \mathcal N_{C(t)}(X_{\varepsilon}(t)) \textrm{ $|DY_{\varepsilon}|$-a.e. with }Y_{\varepsilon}(0) = x_0
\end{array}
\right.
\textrm{$;$ }t\in [0,T],
\end{equation}
where $b :\mathbb R\rightarrow\mathbb R$ is a Lipschitz continuous function, $\varepsilon > 0$, $B$ is a fractional Brownian motion of Hurst index $H\in ]1/2,1[$, $\dot Y_{\varepsilon}$ is the Radon-Nikodym derivative of the differential measure $DY_{\varepsilon}$ of $Y_{\varepsilon}$ with respect to its variation measure $|DY_{\varepsilon}|$, the multifunction $C : [0,T]\rightrightarrows\mathbb R$ is Lipschitz continuous for the Hausdorff distance, $x_0\in C(0)$ and $\mathcal N_{C(t)}(X_{\varepsilon}(t))$ is the normal cone of $C(t)$ at point $X_{\varepsilon}(t)$. The definition of the normal cone is stated later.
\\
\\
A solution to Problem (\ref{main_equation}), if it exists, is a couple $(X_{\varepsilon},Y_{\varepsilon})$ of continuous functions from $[0,T]$ into $\mathbb R$ such that $X_{\varepsilon}(t)\in C(t)$ for every $t\in [0,T]$. Roughly speaking, $X_{\varepsilon}$ coincides with the solution to $dX_{\varepsilon}^{*}(t) = b(X_{\varepsilon}^{*}(t))dt +\varepsilon dB(t)$, except when $X_{\varepsilon}$ hits the frontier of $C$. Each time this situation occurs, $X_{\varepsilon}$ is pushed inside of $C$ with a minimal force by $Y_{\varepsilon}$. The differential inclusion defining the process $Y_{\varepsilon}$ in Problem (\ref{main_equation}) is equivalent to a (Moreau) sweeping process. Several authors studied Problem (\ref{main_equation}) when $H = 1/2$. For instance, the reader can refer to Bernicot and Venel \cite{BV11}, Slominski and Wojciechowski \cite{SW11} or Castaing et al. \cite{CMR15}. When $H\not= 1/2$, the reader can refer to Falkowski and Slominski \cite{FS15} or Castaing et al. \cite{CMR}. In this last paper, the authors proved the existence, uniqueness and the convergence of an approximation scheme of the solution to Problem (\ref{main_equation}) under a \textit{nonempty interior} condition on $C$ (see Assumption \ref{assumption_C}). In fact, in all these papers, the authors studied the Skorokhod reflection problem defined by a SDE and a sweeping process for a multiplicative and/or multidimensional noise.
\\
\\
Let $K :\mathbb R\rightarrow\mathbb R_+$ be a kernel. The paper deals with the consistency, a rate of convergence and the asymptotic distribution of the nonparametric estimator
\begin{displaymath}
\widehat\tau_\varepsilon(t) :=
\frac{1}{h_{\varepsilon}}
\int_{0}^{t}
\int_{0}^{T}K\left(\frac{s - u}{h_{\varepsilon}}\right)dX_{\varepsilon}(s)du
\textrm{ $;$ }
t\in [0,T]
\end{displaymath}
of the trend
\begin{displaymath}
\tau(.) :=\int_{0}^{.}b(x(u))du + y(.) - x_0
\end{displaymath}
of Problem (\ref{main_equation}), where
\begin{equation}\label{main_equation_trend}
\left\{
\begin{array}{rcl}
 x(t) & = & \displaystyle{\int_{0}^{t}b(x(s))ds} + y(t)\\
 -\dot y(t) & \in & \mathcal N_{C(t)}(x(t)) \textrm{ $|Dy|$-a.e. with }y(0) = x_0
\end{array}
\right.
\end{equation}
and $h_{\varepsilon} > 0$ goes to zero when $\varepsilon\rightarrow 0$.
\\
\\
Along the last two decades, many authors studied statistical inference in stochastic differential equations driven by the fractional Brownian motion. Most references on the estimation of the trend component in fractional SDE deals with parametric estimators (see Kleptsyna and Le Breton \cite{KB01}, Tudor and Viens \cite{TV07}, Hu and Nualart \cite{HN10}, Chronopoulou and Tindel \cite{CT13}, Neuenkirch and Tindel \cite{NT14}, Mishura and Ralchenko \cite{MR14}, Hu et al. \cite{HNZ18}, etc.). On the nonparametric estimation of the trend component in fractional SDE, there are only few references. Saussereau \cite{SAUSSEREAU14} and Comte and Marie \cite{CM19} study the consistency of some Nadaraya-Watson's-type estimators of the drift function in a fractional SDE. In \cite{MP11}, Mishra and Prakasa Rao established the consistency and a rate of convergence of a nonparametric estimator of the whole trend of the solution to a fractional SDE. Our paper generalizes their results to the Skorokhod reflection Problem (\ref{main_equation}). On the nonparametric estimation in It\^o's calculus framework, the reader can refer to Kutoyants \cite{KUTOYANTS94} and \cite{KUTOYANTS04}. Up to our knowledge, there is no reference on the nonparametric estimation of the trend in reflected fractional SDE.
\\
\\
Section \ref{section_preliminaries} deals with some preliminaries on the Skorokhod reflection problem defined by a fractional SDE and a sweeping process. Section \ref{section_trend_estimator} deals with the consistency, a rate of convergence and the asymptotic distribution of the estimator $\widehat\tau_{\varepsilon}(t)$.
\\
\\
\textbf{Notations and basic properties:}
\begin{enumerate}
 \item For every $h > 0$, $K_h := 1/hK(\cdot /h)$.
 \item Consider a Hilbert space $(E,\langle .,.\rangle)$. For every closed convex subset $K$ of $E$ and every $x\in E$, $\mathcal N_K(x)$ is the normal cone of $K$ at $x$:
 \begin{displaymath}
 \mathcal N_K(x) :=
 \{y\in E :\forall z\in K\textrm{, }
 \langle y,z - x\rangle\leqslant 0\}.
 \end{displaymath}
 In particular, for $E =\mathbb R$ and $K = [l,u]$ with $l,u\in\mathbb R$ such that $l\leqslant u$,
 \begin{displaymath}
 \mathcal N_K(x) =
 \left\{
 \begin{array}{rcl}
  \mathbb R_- & \textrm{if} & x = l\\
  \mathbb R_+ & \textrm{if} & x = u\\
  \{0\} & \textrm{if} & x\in ]l,u[
 \end{array}\right..
 \end{displaymath}
 \item For every $t\in ]0,T]$, $\Delta_t :=\{(u,v)\in [0,t]^2 : u < v\}$.
 \item For every function $f$ from $[0,T]$ into $\mathbb R$ and $(s,t)\in\Delta_T$, $f(s,t) := f(t) - f(s)$.
 \item Consider $(s,t)\in\Delta_T$. The vector space of continuous functions from $[s,t]$ into $\mathbb R$ is denoted by $C^0([s,t],\mathbb R)$ and equipped with the uniform norm $\|.\|_{\infty,s,t}$ defined by
 \begin{displaymath}
 \|f\|_{\infty,s,t} :=
 \sup_{u\in [s,t]}|f(u)|
 \textrm{ $;$ }
 \forall f\in C^0([s,t],\mathbb R),
 \end{displaymath}
 or the semi-norm $\|.\|_{0,s,t}$ defined by
 \begin{displaymath}
 \|f\|_{0,s,t} :=
 \sup_{u,v\in [s,t]}|f(v) - f(u)|
 \textrm{ $;$ }
 \forall f\in C^0([s,t],\mathbb R).
 \end{displaymath}
 Moreover, $\|.\|_{\infty,T} :=\|.\|_{\infty,0,T}$ and $\|.\|_{0,T} := \|.\|_{0,0,T}$.
 \item Consider $(s,t)\in\Delta_T$. The set of all dissections of $[s,t]$ is denoted by $\mathfrak D_{[s,t]}$.
 \item Consider $(s,t)\in\Delta_T$. A function $f : [s,t]\rightarrow\mathbb R$ is of finite $1$-variation if and only if,
 \begin{displaymath}
 \|f\|_{1\textrm{-var},s,t} :=
 \sup\left\{
 \sum_{k = 1}^{n - 1}|f(t_k,t_{k + 1})|\textrm{ $;$ }
 n\in\mathbb N^*\textrm{ and }
 (t_k)_{k\in\llbracket 1,n\rrbracket}\in\mathfrak D_{[s,t]}\right\}\\
 < \infty.
 \end{displaymath}
 Consider the vector space
 \begin{displaymath}
 C^{1\textrm{-var}}([s,t],\mathbb R) :=
 \{f\in C^0([s,t],\mathbb R) :\|f\|_{1\textrm{-var},s,t} <\infty\}.
 \end{displaymath}
 The map $\|.\|_{1\textrm{-var},s,t}$ is a semi-norm on $C^{1\textrm{-var}}([s,t],\mathbb R)$. Moreover, $\|.\|_{1\textrm{-var},T} :=\|.\|_{1\textrm{-var},0,T}$.
 \item The vector space of Lipschitz continuous functions from a closed interval $I\subset\mathbb R$ into $\mathbb R$ is denoted by $\textrm{Lip}(I)$ and equipped with the Lipschitz semi-norm $\|.\|_{\textrm{Lip},I}$ defined by
 \begin{displaymath}
 \|f\|_{\textrm{Lip},I} :=
 \sup\left\{
 \frac{|f(t) -f(s)|}{|t - s|}
 \textrm{ ; }
 s,t\in I
 \textrm{ and }
 s\not= t\right\}
 \end{displaymath}
 for every $f\in\textrm{Lip}(I)$. Moreover, $\|.\|_{\textrm{Lip}} :=\|.\|_{\textrm{Lip},\mathbb R}$ and $\|.\|_{\textrm{Lip},T} :=\|.\|_{\textrm{Lip},[0,T]}$.
 \item For every $L > 0$,
 \begin{displaymath}
 \Theta_0(L) :=
 \{f\in\textrm{Lip}(\mathbb R) : |f(0)| +\|f\|_{\textrm{Lip}}\leqslant L\}.
 \end{displaymath}
\end{enumerate}
%

% Section : Preliminaries.

%
\section{Preliminaries}\label{section_preliminaries}
This section deals with some preliminaries on the Skorokhod reflection problem defined by a fractional SDE and a sweeping process.
\\
\\
First, the following theorem states a sufficient condition of existence and uniqueness of the solution to the unperturbed sweeping process defined by
\begin{equation}\label{sweeping_process}
\left\{
\begin{array}{rcl}
 -\dot y(t) & \in & \mathcal N_{C(t)}(y(t)) \textrm{ $|Dy|$-a.e.}\\
 y(0) & = & y_0
\end{array}
\right.
\textrm{$;$ }t\in [0,T],
\end{equation}
where $y_0\in C(0)$.
%

% Theorem : Existence and uniqueness of the solution to the unperturbed sweeping process.

%
\begin{theorem}\label{sweeping_process_solution}
Assume that for every $t\in [0,T]$, $C(t)$ is a compact interval of $\mathbb R$. Moreover, assume that there exist $r > 0$ and $a\in\mathbb R$ such that
\begin{displaymath}
[a - r,a + r]\subset\normalfont{\textrm{int}}(C(t))
\textrm{ $;$ }
\forall t\in [0,T].
\end{displaymath}
Then, Problem (\ref{sweeping_process}) has a unique continuous solution of finite $1$-variation $y : [0,T]\rightarrow\mathbb R$ such that
\begin{displaymath}
\|y\|_{1\normalfont{\textrm{-var}},T}
\leqslant
\max\{0,
\|y_0 - a\| - r\}.
\end{displaymath}
\end{theorem}
\noindent
See Monteiro Marques \cite{MONTEIRO_MARQUES84} for a proof.
\\
\\
In the sequel, the multifunction $C$ fulfills the following assumption.
%

% Assumption : Assumption on C.

%
\begin{assumption}\label{assumption_C}
For every $t\in [0,T]$, $C(t)$ is a compact interval of $\mathbb R$. Moreover, there exist $r > 0$ and a continuous selection $\gamma : [0,T]\rightarrow\mathbb R$ such that
\begin{displaymath}
[\gamma(t) - r,\gamma(t) + r]
\subset
\normalfont{\textrm{int}}(C(t))
\textrm{ $;$ }
\forall t\in [0,T].
\end{displaymath}
\end{assumption}
\noindent
Let $\varphi : [0,T]\rightarrow\mathbb R$ be a continuous function such that $\varphi(0) = 0$ and consider the (generic) Skorokhod reflection problem
\begin{equation}\label{Skorokhod_problem}
\left\{
\begin{array}{rcl}
 v_{\varphi}(t) & = & \varphi(t) + w_{\varphi}(t)\\
 -\dot w_{\varphi}(t) & \in & \mathcal N_{C_{\varphi}(t)}(w_{\varphi}(t))
 \textrm{ $|Dw_{\varphi}|$-a.e. with $w_{\varphi}(0) = x_0$}
\end{array}
\right.,
\end{equation}
where
\begin{displaymath}
C_{\varphi}(t) :=\{v -\varphi(t)\textrm{ $;$ }v\in C(t)\}
\textrm{ $;$ }
\forall t\in [0,T],
\end{displaymath}
$v_{\varphi} : [0,T]\rightarrow\mathbb R$ is a continuous function and $w_{\varphi} : [0,T]\rightarrow\mathbb R$ is a continuous function of finite $1$-variation. Under Assumption \ref{assumption_C}, by Theorem \ref{sweeping_process_solution} together with Castaing et al. \cite{CMR16}, Lemma 2.2, Problem (\ref{Skorokhod_problem}) has a unique solution. Moreover, the following proposition provides a suitable control of $w_{\varphi} - w_{\psi}$ for any continuous functions $\varphi,\psi : [0,T]\rightarrow\mathbb R$ such that $\varphi(0) =\psi(0) = 0$.
%

% Proposition : Lipschitz continuity of w_..

%
\begin{proposition}\label{Lipschitz_continuity_w}
Under Assumption \ref{assumption_C}, for every continuous functions $\varphi,\psi : [0,T]\rightarrow\mathbb R$ such that $\varphi(0) =\psi(0) = 0$,
\begin{displaymath}
\|w_{\varphi} - w_{\psi}\|_{\infty,T}\leqslant
\|\varphi -\psi\|_{\infty,T}.
\end{displaymath}
\end{proposition}
\noindent
See Slominski and Wojciechowski \cite{SW11}, Proposition 2.3 for a proof.
\\
\\
Under Assumption \ref{assumption_C}, note that there exist $R > 0$, $N\in\mathbb N^*$ and $(t_0,\dots,t_N)\in\mathfrak D_{[0,T]}$ such that
\begin{displaymath}
[\gamma(t_k) - R,\gamma(t_k) + R]
\subset C(t)
\end{displaymath}
for every $k\in\llbracket 0,N - 1\rrbracket$ and $t\in [t_k,t_{k + 1}]$.
%

% Proposition : Bound of the 1-variation norm of w_{\varphi}.

%
\begin{proposition}\label{bound_w}
Consider $(s,t)\in\Delta_T$ and $\rho\in ]0,R/2]$. Under Assumption \ref{assumption_C}, if $\|\varphi\|_{0,s,t}\leqslant\rho$, then
\begin{displaymath}
\|w_{\varphi}\|_{1\normalfont{\textrm{-var}},s,t}
\leqslant
N\sup_{u\in [0,T]}
\sup_{v,w\in C(u)}
|w - v|.
\end{displaymath}
\end{proposition}
\noindent
The proof of Proposition \ref{bound_w} is the same that the proof of Castaing et al. \cite{CMR}, Proposition 2.5 but with the upper bound for the 1-variation norm of the 1-dimensional unperturbed sweeping process provided in Theorem \ref{sweeping_process_solution} instead of the corresponding upper bound in the multidimensional case provided in Castaing et al. \cite{CMR}, Proposition 2.1.
\\
\\
For any $t\in [0,T]$,
\begin{displaymath}
\mathcal N_{C_{\varphi}(t)}(w_{\varphi}(t)) =
\mathcal N_{C(t) -\varphi(t)}(v_{\varphi}(t) -\varphi(t)) =
\mathcal N_{C(t)}(v_{\varphi}(t)).
\end{displaymath}
Then Problem (\ref{Skorokhod_problem}) is equivalent to
\begin{displaymath}
\left\{
\begin{array}{rcl}
 v_{\varphi}(t) & = & \varphi(t) + w_{\varphi}(t)\\
 -\dot w_{\varphi}(t) & \in & \mathcal N_{C(t)}(v_{\varphi}(t))
 \textrm{ $|Dw_{\varphi}|$-a.e. with $w_{\varphi}(0) = x_0$}
\end{array}
\right..
\end{displaymath}
So, one can use the previous results of this section in order to establish the existence and uniqueness of the solution to Problems (\ref{main_equation}) and (\ref{main_equation_trend}).
%

% Theorem : Existence and uniqueness of the solution to the main equation.

%
\begin{theorem}\label{existence_uniqueness_main_equation}
Under Assumption \ref{assumption_C},
\begin{enumerate}
 \item Problem (\ref{main_equation}) has a unique solution $(X_{\varepsilon},Y_{\varepsilon})$. Moreover, its paths belong to
 \begin{displaymath}
 C^{p\normalfont{\textrm{-var}}}([0,T],\mathbb R)\times C^{1\normalfont{\textrm{-var}}}([0,T],\mathbb R)
 \end{displaymath}
 for every $p > 1/H$.
 \item Problem (\ref{main_equation_trend}) has a unique solution $(x,y)$. Moreover, it is a Lipschitz continuous map from $[0,T]$ into $\mathbb R^2$ such that
 \begin{displaymath}
 \|y\|_{\normalfont{\textrm{Lip}},T}\leqslant
 \|b\|_{\normalfont{\textrm{Lip}}} +
 \|C\|_{\normalfont{\textrm{Lip}},T}
 \end{displaymath}
 and
 \begin{displaymath}
 \|x\|_{\normalfont{\textrm{Lip}},T}\leqslant
 2\|b\|_{\normalfont{\textrm{Lip}}} +
 \|C\|_{\normalfont{\textrm{Lip}},T}.
 \end{displaymath}
\end{enumerate}
\end{theorem}
\noindent
The proof of the existence of solutions to Problem (\ref{main_equation}) in Theorem \ref{existence_uniqueness_main_equation} is the same that the proof of Castaing et al. \cite{CMR}, Theorem 3.1 but with the upper bound for the 1-variation norm of $w_{\varphi}$ in Problem (\ref{Skorokhod_problem}) provided in Proposition \ref{bound_w} instead of the corresponding upper bound in the multidimensional case provided in Castaing et al. \cite{CMR}, Proposition 2.5. Castaing et al. \cite{CMR}, Proposition 4.1 gives the uniqueness of the solution to Problem (\ref{main_equation}). Castaing et al. \cite{CMR16}, Theorem 4.2 gives the existence, uniqueness and the regularity of the solution to Problem (\ref{main_equation_trend}).
%

% Section : Convergence of the trend estimator.

%
\section{Convergence of the trend estimator}\label{section_trend_estimator}
This section deals with the consistency, a rate of convergence and the asymptotic distribution of the estimator $\widehat\tau_{\varepsilon}(t)$. First, the following lemma deals with the convergence of $X_{\varepsilon}$ and $Y_{\varepsilon}$ when $\varepsilon\rightarrow 0$.
%

% Lemma : Continuity of the solution with respect to the volatility constant.

%
\begin{lemma}\label{continuity_volatility}
Under Assumption \ref{assumption_C}, if $b\in\Theta_0(L)$ with $L > 0$, then there exists a deterministic constant $\mathfrak c_{H,L,T} > 0$, depending only on $H$, $L$ and $T$, such that
\begin{displaymath}
\mathbb E(
\|X_{\varepsilon} - x\|_{\infty,T}^{2}) +
\mathbb E(
\|Y_{\varepsilon} - y\|_{\infty,T}^{2})
\leqslant
\mathfrak c_{H,L,T}
\varepsilon^2.
\end{displaymath}
\end{lemma}
%

% Proof.

%
\begin{proof}
Consider $H_{\varepsilon} := X_{\varepsilon} - Y_{\varepsilon}$ and $h := x - y$. By Proposition \ref{Lipschitz_continuity_w}, for any $t\in [0,T]$,
\begin{displaymath}
\|Y_{\varepsilon} - y\|_{\infty,t}\leqslant
\|H_{\varepsilon} - y\|_{\infty,t}.
\end{displaymath}
Then,
\begin{eqnarray*}
 |X_{\varepsilon}(t) - x(t)|
 & \leqslant &
 \|H_{\varepsilon} - h\|_{\infty,t} +
 \|Y_{\varepsilon} - y\|_{\infty,t}
 \leqslant
 2\|H_{\varepsilon} - h\|_{\infty,t}\\
 & \leqslant &
 2L\int_{0}^{t}|X_{\varepsilon}(s) - x(s)|ds + 2\varepsilon\|B\|_{\infty,t}.
\end{eqnarray*}
By Gronwall's lemma,
\begin{displaymath}
|X_{\varepsilon}(t) - x(t)|
\leqslant 2\varepsilon
\|B\|_{\infty,T}e^{2LT}.
\end{displaymath}
Moreover,
\begin{eqnarray*}
 |Y_{\varepsilon}(t) - y(t)|
 & \leqslant &
 |H_{\varepsilon}(t) - h(t)| +
 |X_{\varepsilon}(t) - x(t)|\\
 & \leqslant &
 (TL + 1)\|X_{\varepsilon} - x\|_{\infty,T}
 +\varepsilon\|B\|_{\infty,T}\\
 & \leqslant &
 \varepsilon\|B\|_{\infty,T}
 (2e^{2LT}(TL + 1) + 1).
\end{eqnarray*}
This concludes the proof because $\mathbb E(\|B\|_{\infty,T}^{2}) <\infty$.
\end{proof}
\noindent
In the sequel, the bandwidth $h_{\varepsilon}$ and the kernel $K$ fulfill the following assumptions.
%

% Assumption : Assumption on h.

%
\begin{assumption}\label{assumption_h}
The bandwidth $h_{\varepsilon}$ satisfies $\varepsilon = o(h_{\varepsilon}^{1 - H})$.
\end{assumption}
%

% Assumption : Assumption on K.

%
\begin{assumption}\label{assumption_K}
The kernel $K$ is bounded and $K^{-1}(\{0\})^c = ]A,B[$ with $A < B$.
\end{assumption}
\noindent
For instance, the triangular kernel
\begin{displaymath}
u\in\mathbb R\longmapsto
(1 - |u|)\mathbf 1_{|u|\leqslant 1}
\end{displaymath}
or the parabolic kernel
\begin{displaymath}
u\in\mathbb R\longmapsto
\frac{3}{4}(1 - u^2)\mathbf 1_{|u|\leqslant 1}
\end{displaymath}
fulfill Assumption \ref{assumption_K}.
\\
\\
Let us now establish the consistency and a rate of convergence for the estimator $\widehat\tau_{\varepsilon}$ of the trend $\tau$ of Problem (\ref{main_equation}).
%

% Theorem : Convergence of the trend estimator.

%
\begin{theorem}\label{convergence_trend_estimator}
Under Assumptions \ref{assumption_C} and \ref{assumption_K}, if $b\in\Theta_0(L)$ with $L > 0$, then there exists a deterministic constant $\mathfrak c_{C,H,K,L,T} > 0$, depending only on $C$, $H$, $K$, $L$ and $T$, such that
\begin{displaymath}
\sup_{t\in [0,T]}\mathbb E(|\widehat\tau_{\varepsilon}(t) -\tau(t)|^2)
\leqslant
\mathfrak c_{C,H,K,L,T}(\varepsilon^2 + h_{\varepsilon}^{2} +\varepsilon^2h_{\varepsilon}^{2H - 2}).
\end{displaymath}
In particular, under Assumption \ref{assumption_h}, the estimator $\widehat\tau_{\varepsilon}$ is consistent.
\end{theorem}
%

% Proof.

%
\begin{proof}
First of all, for any $t\in [0,T]$,
\begin{eqnarray*}
 \widehat\tau_{\varepsilon}(t) -\tau(t) & = &
 \int_{0}^{t}\int_{0}^{T}K_{h_{\varepsilon}}(s - u)dX_{\varepsilon}(s)du -
 \int_{0}^{t}b(x(u))du - y(t) + x_0\\
 & = &
 \alpha_{\varepsilon}(t) +
 \beta_{\varepsilon}(t) +
 \gamma_{\varepsilon}(t) +
 \zeta_{\varepsilon}(t) +
 \eta_{\varepsilon}(t),
\end{eqnarray*}
where
\begin{eqnarray*}
 \alpha_{\varepsilon}(t) & := &
 \int_{0}^{t}\int_{0}^{T}K_{h_{\varepsilon}}(s - u)(b(X_{\varepsilon}(s)) - b(x(s)))dsdu,\\
 \beta_{\varepsilon}(t) & := &
 \int_{0}^{t}\int_{0}^{T}K_{h_{\varepsilon}}(s - u)b(x(s))dsdu -
 \int_{0}^{t}b(x(u))du,\\
 \gamma_{\varepsilon}(t) & := &
 \varepsilon\int_{0}^{t}\int_{0}^{T}K_{h_{\varepsilon}}(s - u)dB(s)du,\\
 \zeta_{\varepsilon}(t) & := &
 \int_{0}^{t}\int_{0}^{T}K_{h_{\varepsilon}}(s - u)d(Y_{\varepsilon} - y)(s)du
 \textrm{ and }\\
 \eta_{\varepsilon}(t) & := &
 \int_{0}^{t}\int_{0}^{T}K_{h_{\varepsilon}}(s - u)dy(s)du - y(t) + x_0.
\end{eqnarray*}
Let us find suitable controls of the supremum on $[0,T]$ of the second order moment of all these components.
\begin{itemize}
 \item Note that
 \begin{eqnarray*}
  |\alpha_{\varepsilon}(t)| & = &
  \left|
  \int_{0}^{t}\int_{-u/h_{\varepsilon}}^{(T - u)/h_{\varepsilon}}K(s)(b(X_{\varepsilon}(h_{\varepsilon}s + u)) - b(x(h_{\varepsilon}s + u)))dsdu
  \right|\\
  & \leqslant &
  \|b\|_{\textrm{Lip}}\int_{0}^{t}
  \sup_{0\leqslant h_{\varepsilon}s + u\leqslant T}|X_{\varepsilon}(h_{\varepsilon}s + u) - x(h_{\varepsilon}s + u)|du
  \leqslant
  LT\|X_{\varepsilon} - x\|_{\infty,T}.
 \end{eqnarray*}
 Then, by Lemma \ref{continuity_volatility},
 \begin{displaymath}
 \sup_{t\in [0,T]}
 \mathbb E(\alpha_{\varepsilon}(t)^2)
 \leqslant
 L^2T^2\mathfrak c_{H,L,T}\varepsilon^2.
 \end{displaymath}
 \item Since $C$ is a Lipschitz continuous and compact-valued multifunction, $x$ is bounded by a deterministic constant $M > 0$ depending only on $C$ (not on $b$). Then, 
 \begin{eqnarray*}
  |\beta_{\varepsilon}(t)| & = &
  \left|
  \int_{0}^{T}b(x(s))\int_{(s - t)/h_{\varepsilon}}^{s/h_{\varepsilon}}
  K(u)duds
  -\int_{0}^{t}b(x(u))du\right|\\
  & = &
  \left|
  \int_{-\infty}^{\infty}
  K(u)\int_{0}^{T}
  b(x(s))\mathbf 1_{[h_{\varepsilon}u,h_{\varepsilon}u + t]}(s)dsdu
  -\int_{0}^{t}b(x(s))ds\right|\\
  & = &
  \left|\int_{A}^{B}K(u)\left(
  \int_{0\vee (h_{\varepsilon}u)}^{T\wedge (h_{\varepsilon}u + t)}
  b(x(s))ds - \int_{0}^{t}b(x(s))ds\right)du\right|\\
  & \leqslant &
  2h_{\varepsilon}
  \sup_{z\in [-M,M]}|b(z)|
  \int_{A}^{B}K(u)|u|du.
 \end{eqnarray*}
 Moreover, since $|b(0)| +\|b\|_{\textrm{Lip}}\leqslant L$,
 \begin{displaymath}
 |\beta_{\varepsilon}(t)|
 \leqslant
 2(|A|\vee|B|)Lh_{\varepsilon}.
 \end{displaymath}
 \item By Memin et al. \cite{MMV01}, Theorem 1.1, there exists a deterministic constant $c_1 > 0$, only depending on $H$, such that
 \begin{eqnarray*}
  \mathbb E(\gamma_{\varepsilon}(t)^2) & \leqslant &
  \varepsilon^2t\int_{0}^{t}
  \mathbb E\left(\left|\int_{0}^{T}K_{h_{\varepsilon}}(s - u)dB(s)\right|^2\right)du\\
  & \leqslant &
  c_1\frac{\varepsilon^2T}{h_{\varepsilon}^2}\int_{0}^{T}\left|\int_{0}^{T}K\left(\frac{s - u}{h_{\varepsilon}}\right)^{1/H}ds\right|^{2H}du
  \leqslant
  c_2\varepsilon^2h_{\varepsilon}^{2H - 2},
 \end{eqnarray*}
 where
 \begin{displaymath}
 c_2 :=
 c_1T^2\left|\int_{A}^{B}K(s)^{1/H}ds\right|^{2H}.
 \end{displaymath}
 \item Since the paths of $Y_{\varepsilon} - y$ are continuous and of finite $1$-variation,
 \begin{eqnarray*}
  \zeta_{\varepsilon}(t) & = &
  \int_{0}^{T}\int_{0}^{t}K_{h_{\varepsilon}}(s - u)dud(Y_{\varepsilon} - y)(s)\\
  & = &
  \int_{0}^{T}\int_{(s - t)/h_{\varepsilon}}^{s/h_{\varepsilon}}K(u)dud(Y_{\varepsilon} - y)(s)\\
  & = &
  \int_{-\infty}^{\infty}K(u)\int_{0}^{T}
  \mathbf 1_{[h_{\varepsilon}u,h_{\varepsilon}u + t]}(s)d(Y_{\varepsilon} - y)(s)du\\
  & = &
  \int_{A}^{B}
  K(u)(Y_{\varepsilon} - y)(0\vee (h_{\varepsilon}u),T\wedge (h_{\varepsilon}u + t))du.
 \end{eqnarray*}
 Then, by Lemma \ref{continuity_volatility},
 \begin{displaymath}
 \sup_{t\in [0,T]}
 \mathbb E(\zeta_{\varepsilon}(t)^2)
 \leqslant
 \mathbb E(\|Y_{\varepsilon} - y\|_{\infty,T}^{2})
 \leqslant\mathfrak c_{H,L,T}\varepsilon^2.
 \end{displaymath}
 \item Since $y$ is a Lipschitz continuous function (see Theorem \ref{existence_uniqueness_main_equation}.(2)),
 \begin{eqnarray*}
  |\eta_{\varepsilon}(t)| & = &
  \left|\int_{0}^{T}\int_{0}^{t}K_{h_{\varepsilon}}(s - u)dudy(s) - y(t) + x_0\right|\\
  & \leqslant &
  \int_{A}^{B}
  K(u)|y(0\vee(h_{\varepsilon}u),T\wedge(h_{\varepsilon}u + t)) - y(0,t)|du
  \leqslant
  2(|A|\vee|B|)\|y\|_{\textrm{Lip},T}h_{\varepsilon}.
 \end{eqnarray*}
 Moreover, since $\|y\|_{\textrm{Lip},T}\leqslant L +\|C\|_{\textrm{Lip},T}$,
 \begin{displaymath}
 |\eta_{\varepsilon}(t)|
 \leqslant
 2(|A|\vee |B|)(L +\|C\|_{\textrm{Lip},T})h_{\varepsilon}.
 \end{displaymath}
\end{itemize}
\end{proof}
\noindent
Theorem \ref{convergence_trend_estimator} says that the quadratic risk of the estimator $\widehat\tau_{\varepsilon}(t)$ involves a squared bias of order $\varepsilon^2 + h_{\varepsilon}^{2}$ and a variance term of order $\varepsilon^2h_{\varepsilon}^{2H - 2}$. The best possible rate $\varepsilon^{2/(2 - H)}$ is reached for a bandwidth choice of order $\varepsilon^{1/(2 - H)}$.
%

% Corollary : Convergence of the trend estimator for the optimal bandwidth.

%
\begin{corollary}\label{convergence_trend_estimator_optimal_bandwidth}
Under Assumptions \ref{assumption_C} and \ref{assumption_K}, if $h_{\varepsilon} =\varepsilon^{1/(2 - H)}$, then
\begin{displaymath}
\lim_{\varepsilon\rightarrow 0}
\varepsilon^{\alpha - 2/(2 - H)}\sup_{t\in [0,T]}
\mathbb E(|\widehat\tau_{\varepsilon}(t) -\tau(t)|^2) = 0
\textrm{ $;$ }
\forall\alpha > 0.
\end{displaymath}
\end{corollary}
\noindent
Corollary \ref{convergence_trend_estimator_optimal_bandwidth} is a straightforward consequence of Theorem \ref{convergence_trend_estimator}.
\\
\\
In the sequel, $C$ fulfills the following assumption.
%

% Assumption : Reinforced assumption on C.

%
\begin{assumption}\label{assumption_C_reinforced}
There exist $\mathbf l,\mathbf u\in C^1([0,T],\mathbb R)$ such that for every $t\in [0,T]$, $\mathbf l(t) <\mathbf u(t)$ and
\begin{displaymath}
C(t) = [\mathbf l(t),\mathbf u(t)].
\end{displaymath}
\end{assumption}
\noindent
Finally, Proposition \ref{asymptotic_distribution_trend_estimator} provides the asymptotic distribution of the estimator $\widehat\tau_{\varepsilon}(t)$ for every $t\in\mathcal E :=\mathcal E_{\mathbf l}\cup\mathcal E_{\mathbf u}\cup\mathcal E_{\textrm{int}(C)}$, where
\begin{displaymath}
\mathcal E_I :=
\{s\in [0,T] :\exists\varepsilon > 0\textrm{, }\forall r\in ]s -\varepsilon,s +\varepsilon[
\textrm{, }x(r)\in I(r)\}
\end{displaymath}
for every multifunction $I : [0,T]\rightrightarrows\mathbb R$.
\\
\\
First, recall that for any $f\in C^{1\textrm{-var}}([0,T],\mathbb R)$, $\dot f$ is the Radon-Nikodym derivative of the differential measure $Df$ of $f$ with respect to its variation measure $|Df|$. In particular, if $f$ is absolutely continuous, then
\begin{displaymath}
f(v) - f(u) =\int_{u}^{v}\dot f(s)ds\textrm{ $;$ }
\forall (u,v)\in\Delta_T.
\end{displaymath}
%

% Lemma : Continuity of the derivative of y.

%
\begin{lemma}\label{continuity_derivative_y}
Under Assumption \ref{assumption_C_reinforced}, $\dot y$ is continuous on $\mathcal E$.
\end{lemma}
%

% Proof.

%
\begin{proof}
Since $y$ is a Lipschitz continuous function, it is absolutely continuous. In other words, for every $(u,v)\in\Delta_T$,
\begin{displaymath}
y(v) - y(u) =\int_{u}^{v}\dot y(s)ds.
\end{displaymath}
On the one hand, consider $t\in\mathcal E_{\textrm{int}(C)}$. So, there exists $\varepsilon > 0$ such that for any $s\in ]t -\varepsilon,t +\varepsilon[$, $x(s)\in]\mathbf l(s),\mathbf u(s)[$ and then
\begin{displaymath}
\dot y(s) = 0.
\end{displaymath}
Therefore, $\dot y$ is continuous at time $s$. On the other hand, consider $t\in\mathcal E_{\mathbf l}$. So, there exists $\varepsilon > 0$ such that for any $s\in ]t -\varepsilon,t +\varepsilon[$, $x(s) =\mathbf l(s)$ and then
\begin{displaymath}
\dot y(s) =\dot{\mathbf l}(s) - b(\mathbf l(s)).
\end{displaymath}
Therefore, since $\mathbf l\in\textrm C^1([0,T],\mathbb R)$, $\dot y$ is continuous at time $s$. The same idea gives the continuity of $\dot y$ on $\mathcal E_{\mathbf u}$.
\end{proof}
\noindent
The previous lemma states that $\dot y$ is continuous when $x$ stays a little time on the frontier or in the interior of $C$. Unfortunately, there is no reason for $\dot y$ to be continuous each time $x$ enters or exits the frontier of $C$.
%

% Proposition : Asymptotic distribution of the trend estimator.

%
\begin{proposition}\label{asymptotic_distribution_trend_estimator}
Under Assumptions \ref{assumption_C_reinforced} and \ref{assumption_K}, if $A\geqslant 0$, $t\in\mathcal E\cap [0,T[$ and $h_{\varepsilon} =\varepsilon^{1/(2 - H)}$, then
\begin{displaymath}
\varepsilon^{-1/(2 - H)}
(\widehat\tau_{\varepsilon}(t) -\tau(t) -\gamma_{\varepsilon}(t))
\xrightarrow[\varepsilon\rightarrow 0]{\mathbb L^2}
\mu(t)
\end{displaymath}
and
\begin{displaymath}
\varepsilon^{-1/(2 - H)}\dot\gamma_{\varepsilon}(t)
\xrightarrow[\varepsilon\rightarrow 0]{\Delta}\mathcal N(0,\sigma_{H,K}^{2}),
\end{displaymath}
where
\begin{displaymath}
\mu(t) :=
(b(x(t)) - b(x(0)) +\dot y(t) -\dot y(0))\int_{A}^{B}K(u)udu
\end{displaymath}
and
\begin{displaymath}
\sigma_{H,K}^{2} :=
H(2H - 1)
\int_{A}^{B}\int_{A}^{B}
|u - v|^{2H - 2}K(u)K(v)dudv.
\end{displaymath}
\end{proposition}
%

% Proof.

%
\begin{proof}
Since
\begin{displaymath}
\sup_{t\in [0,T]}\mathbb E(\alpha_{\varepsilon}(t)^2) +
\sup_{t\in [0,T]}\mathbb E(\zeta_{\varepsilon}(t)^2) = O(\varepsilon^2)
\end{displaymath}
as established in the proof of Theorem \ref{convergence_trend_estimator},
\begin{displaymath}
\varepsilon^{-1/(2 - H)}(\alpha_{\varepsilon}(t) +
\zeta_{\varepsilon}(t))
\xrightarrow[\varepsilon\rightarrow 0]{\mathbb L^2} 0.
\end{displaymath}
Let us study the behaviour of $\varepsilon^{-1/(2 - H)}(\beta_{\varepsilon}(t) +\eta_{\varepsilon}(t))$ when $\varepsilon\rightarrow 0$.
\begin{itemize}
 \item Since $A\geqslant 0$ and $h_{\varepsilon}B + t < T$ for $\varepsilon$ small enough,
 \begin{eqnarray*}
  \beta_{\varepsilon}(t) & = &
  \int_{A}^{B}K(u)\left(
  \int_{t}^{h_{\varepsilon}u + t}
  b(x(s))ds - \int_{0}^{h_{\varepsilon}u}b(x(s))ds\right)du\\
  & = &
  h_{\varepsilon}
  \int_{A}^{B}K(u)u\left(
  \int_{0}^{1}
  b(x(sh_{\varepsilon}u + t))ds - \int_{0}^{1}b(x(sh_{\varepsilon}u))ds\right)du.
 \end{eqnarray*}
 Therefore, by Lebesgue's theorem,
 \begin{displaymath}
 \lim_{\varepsilon\rightarrow 0}
 \varepsilon^{-1/(2 - H)}
 \beta_{\varepsilon}(t) =
 (b(x(t)) - b(x(0)))\int_{A}^{B}K(u)udu.
 \end{displaymath}
 \item Since $y$ is a Lipschitz continuous function, as recalled previously, $dy(s) =\dot y(s)ds$. Then,
 \begin{eqnarray*}
  \eta_{\varepsilon}(t) & = &
  \int_{A}^{B}K(u)(y(t,h_{\varepsilon}u + t) - y(0,h_{\varepsilon}u))du\\
  & = &
  h_{\varepsilon}\int_{A}^{B}K(u)u\left(\int_{0}^{1}\dot y(sh_{\varepsilon}u + t)ds -
  \int_{0}^{1}\dot y(sh_{\varepsilon}u)ds\right)du.
 \end{eqnarray*}
 Therefore, since $\dot y$ is continuous on a neighborhood of $t$ by Lemma \ref{continuity_derivative_y}, by Lebesgue's theorem,
 \begin{displaymath}
 \lim_{\varepsilon\rightarrow 0}
 \varepsilon^{-1/(2 - H)}
 \eta_{\varepsilon}(t) =
 (\dot y(t) -\dot y(0))\int_{A}^{B}K(u)udu.
 \end{displaymath}
\end{itemize}
Finally,
\begin{displaymath}
\dot\gamma_{\varepsilon}(t) =
\varepsilon
\int_{0}^{T}K_{h_{\varepsilon}}(s - t)dB(s)
\rightsquigarrow\mathcal N(0,\sigma_{\varepsilon}(t)^2)
\end{displaymath}
where
\begin{eqnarray*}
 \sigma_{\varepsilon}(t)^2 & := &
 H(2H - 1)\varepsilon^2\int_{0}^{T}\int_{0}^{T}
 |s - r|^{2H - 2}K_{h_{\varepsilon}}(r - t)K_{h_{\varepsilon}}(s - t)drds\\
 & = &
 \sigma_{H,K}\varepsilon^2h_{\varepsilon}^{2H - 2}.
\end{eqnarray*}
Therefore,
\begin{displaymath}
\varepsilon^{-1/(2 - H)}\dot\gamma_{\varepsilon}(t)
\xrightarrow[\varepsilon\rightarrow 0]{\Delta}\mathcal N(0,\sigma_{H,K}^{2}).
\end{displaymath}
\end{proof}
\noindent
Since $x$ is Lipschitz continuous on $[0,T]$, the subset of times $x$ enters or exists the frontier of $C$ is countable. So, the Lebesgue measure of $\mathcal E$ is equal to $T$. Therefore, Proposition \ref{asymptotic_distribution_trend_estimator} is true for almost every $t$ in $[0,T]$.
%

% References.

%

%
\end{document}